\documentclass[12pt,reqno]{article}

\usepackage[usenames]{color}
\usepackage{amssymb}
\usepackage{graphicx}
\usepackage{amscd}

\usepackage{amsthm}
\newtheorem{theorem}{Theorem}
\newtheorem{lemma}[theorem]{Lemma}
\newtheorem{proposition}[theorem]{Proposition}
\newtheorem{corollary}[theorem]{Corollary}

\theoremstyle{definition}

\newtheorem{example}[theorem]{Example}

\usepackage[colorlinks=true,
linkcolor=webgreen, filecolor=webbrown,
citecolor=webgreen]{hyperref}

\definecolor{webgreen}{rgb}{0,.5,0}
\definecolor{webbrown}{rgb}{.6,0,0}

\usepackage{color}

\usepackage{float}

\usepackage{graphics,amsmath,amssymb}
\usepackage{amsfonts}
\usepackage{latexsym}
\usepackage{epsf}

\newcommand{\seqnum}[1]{\href{http://www.research.att.com/cgi-bin/access.cgi/as/~njas/sequences/eisA.cgi?Anum=#1}{\underline{#1}}}

\begin{document}

\begin{center}
\vskip 1cm{\LARGE\bf A Note on Riordan Arrays with Catalan Halves} \vskip 1cm \large
Paul Barry\\
School of Science\\
Waterford Institute of Technology\\
Ireland\\
\href{mailto:pbarry@wit.ie}{\tt pbarry@wit.ie}
\end{center}
\vskip .2 in

\begin{abstract} For a lower triangular matrix $(t_{n,k})$ we call the matrices with respective entries $(t_{2n-k,n})$ and $(t_{2n,n+k})$ the vertical and the horizontal halves. In this note, we discuss Riordan arrays whose halves are closely related to the Catalan matrices. \end{abstract}

\section{Preliminaries on Riordan arrays}
We recall some facts about Riordan arrays in this introductory section. Readers familiar with Riordan arrays may wish to move on to the next section. 

A Riordan array \cite{Book, SGWW} is defined by a pair of power series
$$g(x)=g_0 + g_1 x + g_2 x^2 + \cdots=\sum_{n=0}^{\infty} g_n x^n,$$ and
$$f(x)=f_1 x + f_2 x^2+ f_3 x^3 + \cdots = \sum_{n=1}^{\infty} f_n x^n.$$
We require that $g_0 \ne 0$ (and hence $g(x)$ is invertible, with inverse $\frac{1}{g(x)}$), while we also demand that
$f_0=0$ and $f_1 \ne 0$ (hence $f(x)$ has a compositional inverse $\bar{f}(x)=\text{Rev}(f)(x)$ defined by $f(\bar{f}(x))=x$).
The set of such pairs $(g(x), f(x))$ forms a group (called the Riordan group \cite{SGWW}) with multiplication
$$(g(x), f(x)) \cdot (u(x), v(x))=(g(x)u(f(x)), v(f(x)),$$ and with inverses given by
$$(g(x), f(x))^{-1}=\left(\frac{1}{g(\bar{f}(x))}, \bar{f}(x)\right).$$
The coefficients of the power series may be drawn from any ring (for example, the integers $\mathbb{Z}$) where these operations make sense. To each such ring there exists a corresponding Riordan group.

There is a matrix representation of this group, where to the element $(g(x), f(x))$ we associate the matrix
$\left(a_{n,k}\right)_{0 \le n,k \le \infty}$ with general element
$$t_{n,k}=[x^n] g(x)f(x)^k.$$
Here, $[x^n]$ is the functional that extracts the coefficient of $x^n$ in a power series \cite{MC}. In this representation, the group law corresponds to ordinary matrix multiplication, and the inverse of $(g(x), f(x))$ is represented by the inverse of $\left(t_{n,k}\right)_{0 \le n,k \le \infty}$.

The Fundamental Theorem of Riordan arrays is the rule
$$(g(x), f(x))\cdot h(x)=g(x)h(f(x)),$$ detailing how an array $(g(x), f(x))$ can act on a power series. This corresponds to the matrix $(t_{n,k})$ multiplying the vector $(h_0, h_1, h_2,\ldots)^T$.

\begin{example} Pascal's triangle, also known as the binomial matrix, is defined by the Riordan group element
$$B=\left(\frac{1}{1-x}, \frac{x}{1-x}\right).$$ This means that we have
$$\binom{n}{k}=[x^n] \frac{1}{1-x} \left(\frac{x}{1-x}\right)^k.$$
To see that this is so, we need to be familiar with the rules of operation of the functional $[x^n]$ \cite{MC}.
We have
\begin{align*}
[x^n] \frac{1}{1-x} \left(\frac{x}{1-x}\right)^k&=[x^n] \frac{x^k}{(1-x)^{k+1}}\\
&= [x^{n-k}] (1-x)^{-(k+1)}\\
&= [x^{n-k}] \sum_{j=0}^{\infty} \binom{-(k+1)}{j}(-1)^j x^j\\
&= [x^{n-k}] \sum_{j=0}^{\infty} \binom{k+1+j-1}{j}x^j\\
&= [x^{n-k}] \sum_{j=0}^{\infty} \binom{k+j}{j} x^j \\
&= \binom{k+n-k}{n-k}=\binom{n}{n-k}=\binom{n}{k}.\end{align*}
\end{example}

The binomial matrix is an element of the Bell subgroup of the Riordan group, consisting of arrays of the form $(g(x), xg(x))$. It is also an element of the hitting time subgroup, which consists of arrays of the form
$\left(\frac{x f'(x)}{f(x)}, f(x)\right)$. Arrays of the form $(1, f(x))$ belong to the associated or Lagrange subgroup of the Riordan group.

Note that all the arrays in this note are lower triangular matrices of infinite extent. We show appropriate truncations.

Many examples of sequences and  Riordan arrays are documented in the On-Line Encyclopedia of Integer Sequences (OEIS) \cite{SL1, SL2}. Sequences are frequently referred to by their
OEIS number. For instance, the binomial matrix $\mathbf{B}=\left(\frac{1}{1-x}, \frac{x}{1-x}\right)$ (``Pascal's triangle'') is \seqnum{A007318}. In the sequel we will not distinguish between an array pair $(g(x), f(x))$ and its matrix representation. The Hankel transform of a sequence $a_n$ is the sequence of determinants $h_n=|a_{i+j}|_{i \le i,j \le n}$.

The Catalan numbers $C_n=\frac{1}{n+1} \binom{2n}{n}$ \seqnum{A000108} have generating function
$$c(x)=\frac{1-\sqrt{1-4x}}{2x}.$$
We note that $$\text{Rev}(xc(x)) = x(1-x).$$ The Catalan numbers $C_n$ are the unique numbers such that the sequences $C_n$ and $C_{n+1}$ both have their Hankel tranforms given by $h_n=1$ for all $n \ge 0$. 

There are a number of Riordan arrays that are closely related to the Catalan numbers. The principal ones are $(1, xc(x))$ \seqnum{A106566}, $(1, c(x)-1)=(1,xc(x)^2)$ \seqnum{A128899}, $(c(x), xc(x))$ \seqnum{A033184} and $(c(x)^2, xc(x)^2)$ \seqnum{A039598}. These, and their reverse triangles, are collectively known as Catalan matrices. For instance, the matrix $(1, xc(x))$ begins
$$\left(
\begin{array}{ccccccc}
 1 & 0 & 0 & 0 & 0 & 0 & 0 \\
 0 & 1 & 0 & 0 & 0 & 0 & 0 \\
 0 & 1 & 1 & 0 & 0 & 0 & 0 \\
 0 & 2 & 2 & 1 & 0 & 0 & 0 \\
 0 & 5 & 5 & 3 & 1 & 0 & 0 \\
 0 & 14 & 14 & 9 & 4 & 1 & 0 \\
 0 & 42 & 42 & 28 & 14 & 5 & 1 \\
\end{array}
\right).$$ 
This is \seqnum{A106566}.

\section{The vertical and horizontal halves of a Riordan array}
Given a Riordan array $M=(g(x), f(x))$ with matrix representation $\left(t_{n,k}\right)$ we shall denote by its \emph{vertical half} the matrix $V$ with general $(n,k)$-th term $t_{2n-k,n}$.
We have the following result \cite{Half, Yang3, Yang2, Yang1}.
\begin{lemma}
Given a Riordan array $M=(g(x), f(x))$, its vertical half $V$ is the Riordan array
$$V=\left(\frac{\phi(x)\phi'(x) g(\phi(x))}{f(\phi(x))}, \phi(x)\right)=\left(\frac{x \phi'(x) g(\phi(x))}{\phi(x)}, \phi(x)\right),$$ where
$$\phi(x)=\text{Rev}\left(\frac{x^2}{f(x)}\right).$$
\end{lemma}
\begin{corollary} We have the factorization
$$V=(g(\phi(x)),x) \cdot \left(\frac{x \phi'(x)}{\phi(x)}, \phi(x)\right),$$ where the factor
$\left(\frac{x \phi'(x)}{\phi(x)}, \phi(x)\right)$ is an element of the hitting-time subgroup of the Riordan group.
\end{corollary}
The \emph{horizontal half} $H$ of the array $M=(g(x), f(x))$ is the array whose matrix representation has general $(n,k)$-th term given by $t_{2n,n+k}$ \cite{Half}.
We then have the following result \cite{Half, Central}.
\begin{lemma}Given a Riordan array $(g(x), f(x))=(g(x), xh(x))$, its horizontal half $H$ is the Riordan array
$$\left(\frac{\phi'}{h(\phi)}, \phi\right)\cdot (g, f)=\left(\frac{x \phi'}{\phi}, \phi\right)\cdot (g, f)$$ where
$$\phi(x)=\text{Rev}\left(\frac{x^2}{f(x)}\right).$$
\end{lemma}
\begin{corollary}
We have
$$H=\left(\frac{x \phi' g(\phi)}{\phi}, f(\phi)\right)=\left(\frac{\phi \phi' g(\phi)}{f(\phi)}, f(\phi)\right).$$
\end{corollary}
\begin{corollary}
We have
$$H=(g(\phi(x)),x) \cdot \left(\frac{x \phi'(x)}{\phi(x)}, f(\phi(x))\right).$$
\end{corollary}
\begin{proof}
We have
\begin{align*} H&=\left(\frac{x \phi'}{\phi}, \phi\right)\cdot (g, f)\\
&=\left(\frac{x\phi' g(\phi)}{\phi}, f(\phi)\right)\\
&=(g(\phi), x)\cdot \left(\frac{x \phi'}{\phi}, f(\phi)\right).\end{align*}
\end{proof}
\begin{proposition} Let $V$ and $H$ be respectively the vertical and horizontal halves of the Riordan array $(g(x), f(x))$. Then we have
$$V^{-1} \cdot H = (1, f(x)).$$
\end{proposition}
In general, we have
$$V=(g(\phi), x)\cdot \left(\frac{x \phi'}{\phi}, \phi\right),$$
$$H=(g(\phi), x)\cdot \left(\frac{x \phi'}{\phi}, f(\phi)\right)=(g(\phi), x)\cdot \left(\frac{x \phi'}{\phi}, \phi\right)\cdot (1, f).$$
Thus
$$H = V \cdot (1,f).$$

We have a generic factorization of elements of the associated or Lagrange subgroup of the Riordan group.
\begin{proposition}
Let $A=(1, f(x))$ be an element of the associated group. Let $H$ and $V$ respectively be the horizontal half and the vertical half of $A$. The array $V$ is an element of the hitting time subgroup, and thus so is $V^{-1}$. We then have
$$ V \cdot A = H,$$ or equivalently,
$$ A = V^{-1} \cdot H.$$
\end{proposition}

We note that we can express the inverse $V^{-1}$ of the vertical half $V$ of $(1, f(x))$ in terms of $f(x)$.
\begin{proposition} For the vertical half $V$ of the array $(1, f(x))$ we have
$$ V^{-1}=\left(2-\frac{xf'(x)}{f(x)}, \frac{x^2}{f(x)}\right).$$
\end{proposition}
\begin{proof}
We have $$V=\left(\frac{x \phi'(x)}{\phi(x)}, \phi(x)\right),$$ where
$$\phi(x)=\text{Rev}\left(\frac{x^2}{f(x)}\right).$$
Now since $V$ is in the hitting-time subgroup, its inverse will be given by
$$V^{-1}=\left(\frac{x \bar{\phi}'(x)}{\bar{\phi}(x)}, \bar{\phi}(x)\right).$$
Here, we have
$$\bar{\phi}(x)=\frac{x^2}{f(x)}.$$
We find that
$$\frac{x \bar{\phi}'(x)}{\bar{\phi}(x)}=\frac{2f(x)-xf'(x)}{f(x)},$$ and the result follows.
\end{proof}

It is of interest to calculate the $A$-sequences of the matrices $H$ and $V$. For the horizontal half, we have the following result \cite{Central}.
\begin{proposition} Let the $A$-sequence of the Riordan array $A=(g,f)$ have generating function $A(x)$. Then the generating function of the $A$-sequence of the horizontal half $H$ of $A$ is given by $A(x)^2$.
\end{proposition}
With regard to the vertical half $V$ of $A$, we have the following result.
\begin{proposition} The $A$-sequence of the vertical half of the Riordan array $A=(g,f)$ has generating function $\frac{f(x)}{x}$.
\end{proposition}

We close this section by noting the importance of the hitting-time group element $\left(\frac{x \phi'}{\phi}, \phi\right)$.
$$ H = (g(\phi),x)\cdot \left(\frac{x \phi'}{\phi}, \frac{\phi^2}{x}\right)=\left(\frac{x \phi'}{\phi}, \phi\right)\cdot (g, f).$$
$$V=(g(\phi), x) \cdot \left(\frac{x \phi'}{\phi}, \phi\right).$$
We see that in general we have
$$ V \cdot (g, f)= (g(\phi(x)),x) \cdot H.$$

\section{Catalan rich Riordan arrays}
\begin{example}
In this example, we consider the halves of the Riordan array
$$\left(\frac{1+2x}{1+x}, \frac{-x}{1+x}\right)=\left(\frac{1}{1-x},\frac{-x}{1+x}\right)^{-1}.$$
This matrix begins
$$\left(
\begin{array}{ccccccccccc}
 1 & 0 & 0 & 0 & 0 & 0 & 0 & 0 & 0 & 0 & 0 \\
 1 & -1 & 0 & 0 & 0 & 0 & 0 & 0 & 0 & 0 & 0 \\
 -1 & 0 & 1 & 0 & 0 & 0 & 0 & 0 & 0 & 0 & 0 \\
 1 & 1 & -1 & -1 & 0 & 0 & 0 & 0 & 0 & 0 & 0 \\
 -1 & -2 & 0 & 2 & 1 & 0 & 0 & 0 & 0 & 0 & 0 \\
 1 & 3 & 2 & -2 & -3 & -1 & 0 & 0 & 0 & 0 & 0 \\
 -1 & -4 & -5 & 0 & 5 & 4 & 1 & 0 & 0 & 0 & 0 \\
 1 & 5 & 9 & 5 & -5 & -9 & -5 & -1 & 0 & 0 & 0 \\
 -1 & -6 & -14 & -14 & 0 & 14 & 14 & 6 & 1 & 0 & 0 \\
 1 & 7 & 20 & 28 & 14 & -14 & -28 & -20 & -7 & -1 & 0 \\
 -1 & -8 & -27 & -48 & -42 & 0 & 42 & 48 & 27 & 8 & 1 \\
\end{array}
\right).$$
The general term of this matrix is
$$t_{n,k}=(-1)^n \left(\binom{n}{n-k}-2 \binom{n-1}{n-k-1}\right).$$
We find that the vertical half $V$ of this matrix begins
$$\left(
\begin{array}{ccccccc}
 1 & 0 & 0 & 0 & 0 & 0 & 0 \\
 0 & -1 & 0 & 0 & 0 & 0 & 0 \\
 0 & -1 & 1 & 0 & 0 & 0 & 0 \\
 0 & -2 & 2 & -1 & 0 & 0 & 0 \\
 0 & -5 & 5 & -3 & 1 & 0 & 0 \\
 0 & -14 & 14 & -9 & 4 & -1 & 0 \\
 0 & -42 & 42 & -28 & 14 & -5 & 1 \\
\end{array}
\right),$$ with general term
$$t_{2n-k,n}=(-1)^k \left(\binom{2n-k}{n-k}-2\binom{2n-k-1}{n-k-1}\right).$$

while the horizontal half $H$ begins
$$\left(
\begin{array}{ccccccc}
 1 & 0 & 0 & 0 & 0 & 0 & 0 \\
 0 & 1 & 0 & 0 & 0 & 0 & 0 \\
 0 & 2 & 1 & 0 & 0 & 0 & 0 \\
 0 & 5 & 4 & 1 & 0 & 0 & 0 \\
 0 & 14 & 14 & 6 & 1 & 0 & 0 \\
 0 & 42 & 48 & 27 & 8 & 1 & 0 \\
 0 & 132 & 165 & 110 & 44 & 10 & 1 \\
\end{array}
\right),$$ with general term
$$t_{2n,n+k}=\binom{2n}{n-k}-2\binom{2n-1}{n-k-1}.$$
\begin{proposition} For the Riordan array $\left(\frac{1+2x}{1+x}, \frac{-x}{1+x}\right)$ we have
$$V=(1, -xc(x)) \quad \text{and}\quad H = (1, xc(x)^2).$$
\end{proposition}
\begin{proof}

We have $f(x)=-\frac{x}{1+x}$ and $g(x)=\frac{1+2x}{1+x}$. Thus
$$\frac{x^2}{f(x)}=-x(1+x)$$ and then we have
$$\phi=\text{Rev}\frac{x^2}{f(x)}=\text{Rev}(-x(1+x))=-x c(x).$$
We find that
$$\frac{x \phi'(x)}{\phi(x)}=\frac{1+\sqrt{1-4x}}{2 \sqrt{1-4x}},$$ while
$$g(\phi(x))=\frac{\sqrt{1-4x}(1-\sqrt{1-4x})}{2x}= c(x) \sqrt{1-4x}.$$
But then we have
$$\frac{g(\phi(x))x \phi'(x)}{\phi(x)}= c(x) \sqrt{1-4x} \frac{1+\sqrt{1-4x}}{2 \sqrt{1-4x}}=1.$$
Thus $V=(1, -xc(x))$.
Finally, we have
\begin{align*}
H &=V \cdot (1, f)\\
&=(1, -xc(x)) \cdot \left(1, \frac{-x}{1+x}\right)\\
&=\left(1, \frac{xc(x)}{1-xc(x)}\right)\\
&=(1, xc(x)^2).\end{align*}
\end{proof}
We note further that the matrix with general term $t_{n+k,k}$ begins
$$\left(
\begin{array}{ccccccccc}
 -1 & -1 & 1 & -1 & 1 & -1 & 1 & -1 & 1 \\
 1 & 0 & -1 & 2 & -3 & 4 & -5 & 6 & -7 \\
 -1 & 1 & 0 & -2 & 5 & -9 & 14 & -20 & 27 \\
 1 & -2 & 2 & 0 & -5 & 14 & -28 & 48 & -75 \\
 -1 & 3 & -5 & 5 & 0 & -14 & 42 & -90 & 165 \\
 1 & -4 & 9 & -14 & 14 & 0 & -42 & 132 & -297 \\
 -1 & 5 & -14 & 28 & -42 & 42 & 0 & -132 & 429 \\
 1 & -6 & 20 & -48 & 90 & -132 & 132 & 0 & -429 \\
 -1 & 7 & -27 & 75 & -165 & 297 & -429 & 429 & 0 \\
\end{array}
\right).$$
This square matrix contains a reversed copy of $(c(x),-xc(x))$ and its negative. The diagonal sums of this matrix give the sequence $0^n$, while the diagonal sums of the absolute value of terms gives the sequence \seqnum{A063886}, with generating function $\sqrt{\frac{1+2x}{1-2x}}$. The matrix
$$B \cdot (t_{n+k,k}) \cdot B^T$$ is the matrix that begins
$$\left(
\begin{array}{ccccccccc}
 1 & 0 & 0 & 0 & 0 & 0 & 0 & 0 & 0 \\
 2 & 1 & 0 & 0 & 0 & 0 & 0 & 0 & 0 \\
 2 & 2 & 1 & 0 & 0 & 0 & 0 & 0 & 0 \\
 2 & 2 & 2 & 1 & 0 & 0 & 0 & 0 & 0 \\
 2 & 2 & 2 & 2 & 1 & 0 & 0 & 0 & 0 \\
 2 & 2 & 2 & 2 & 2 & 1 & 0 & 0 & 0 \\
 2 & 2 & 2 & 2 & 2 & 2 & 1 & 0 & 0 \\
 2 & 2 & 2 & 2 & 2 & 2 & 2 & 1 & 0 \\
 2 & 2 & 2 & 2 & 2 & 2 & 2 & 2 & 1 \\
\end{array}
\right).$$
This is the Riordan array $\left(\frac{1+x}{1-x}, x\right)$.
Taken modulo $2$, the matrix $(t_{n+k,k})$ begins
$$\left(
\begin{array}{ccccccccc}
 1 & 1 & 1 & 1 & 1 & 1 & 1 & 1 & 1 \\
 1 & 0 & 1 & 0 & 1 & 0 & 1 & 0 & 1 \\
 1 & 1 & 0 & 0 & 1 & 1 & 0 & 0 & 1 \\
 1 & 0 & 0 & 0 & 1 & 0 & 0 & 0 & 1 \\
 1 & 1 & 1 & 1 & 0 & 0 & 0 & 0 & 1 \\
 1 & 0 & 1 & 0 & 0 & 0 & 0 & 0 & 1 \\
 1 & 1 & 0 & 0 & 0 & 0 & 0 & 0 & 1 \\
 1 & 0 & 0 & 0 & 0 & 0 & 0 & 0 & 1 \\
 1 & 1 & 1 & 1 & 1 & 1 & 1 & 1 & 0 \\
\end{array}
\right),$$ and the diagonal sums of this matrix yield Gould's sequence \seqnum{A001316}, which begins
$$1, 2, 2, 4, 2, 4, 4, 8, 2, 4, 4, 8, 4, 8, 8, 16,\ldots.$$

It is instructive to explore the inverse matrix $\left(\frac{1}{1-x},\frac{-x}{1+x}\right)$ which begins
$$\left(
\begin{array}{ccccccc}
 1 & 0 & 0 & 0 & 0 & 0 & 0 \\
 1 & -1 & 0 & 0 & 0 & 0 & 0 \\
 1 & 0 & 1 & 0 & 0 & 0 & 0 \\
 1 & -1 & -1 & -1 & 0 & 0 & 0 \\
 1 & 0 & 2 & 2 & 1 & 0 & 0 \\
 1 & -1 & -2 & -4 & -3 & -1 & 0 \\
 1 & 0 & 3 & 6 & 7 & 4 & 1 \\
\end{array}
\right).$$
In this case, we have
$$t_{n,k}=\sum_{i=0}^{n-k}\binom{n-i-1}{n-k-i}(-1)^{n-i}.$$
The vertical half of this matrix will then have general its general term given by
$$V_{n,k}=t_{2n-k,n}=\sum_{i=0}^{n-k}\binom{2n-k-i-1}{n-k-i}(-1)^{k+i}.$$
The Riordan array $V$ begins
$$\left(
\begin{array}{ccccccc}
 1 & 0 & 0 & 0 & 0 & 0 & 0 \\
 0 & -1 & 0 & 0 & 0 & 0 & 0 \\
 2 & -1 & 1 & 0 & 0 & 0 & 0 \\
 6 & -4 & 2 & -1 & 0 & 0 & 0 \\
 22 & -13 & 7 & -3 & 1 & 0 & 0 \\
 80 & -46 & 24 & -11 & 4 & -1 & 0 \\
 296 & -166 & 86 & -40 & 16 & -5 & 1 \\
\end{array}
\right).$$
Since $f(x)=-\frac{x}{1+x}$, we find that again we have
$$\phi(x)=-x c(x).$$
Now $g(x)=\frac{1}{1-x}$, so that $g(\phi(x))=\frac{1}{1+xc(x)}$.
Then
$$\frac{g(\phi(x))x\phi'(x)}{\phi(x)}=\frac{1}{1+xc(x)}\frac{1+\sqrt{1-4x}}{2 \sqrt{1-4x}}=\frac{1-x+\sqrt{1-4x}}{\sqrt{1-4x}(x+2)}.$$
We then have that
$$V=\left(\frac{1-x+\sqrt{1-4x}}{\sqrt{1-4x}(x+2)},-xc(x)\right)=\left(\frac{1+x-2x^2}{1+x},-x(1+x)\right)^{-1}.$$
We can factorize $V^{-1}$ as follows.
$$V^{-1}=\left(\frac{1+x-2x^2}{1+x},-x(1+x)\right)=\left(\frac{1+x-2x^2}{1+x},-x\right)\cdot (1, x(1-x)).$$
This gives us the following factorization of $V$.
$$V = (1, xc(x))\cdot \left(\frac{1-x}{1-x-2x^2}, -x\right).$$
The first matrix is a Catalan matrix, and the second matrix is the matrix that begins
$$\left(
\begin{array}{ccccccc}
 1 & 0 & 0 & 0 & 0 & 0 & 0 \\
 0 & -1 & 0 & 0 & 0 & 0 & 0 \\
 2 & 0 & 1 & 0 & 0 & 0 & 0 \\
 2 & -2 & 0 & -1 & 0 & 0 & 0 \\
 6 & -2 & 2 & 0 & 1 & 0 & 0 \\
 10 & -6 & 2 & -2 & 0 & -1 & 0 \\
 22 & -10 & 6 & -2 & 2 & 0 & 1 \\
\end{array}
\right).$$
Here, the numbers
$$1,0,2,2,6,10,22,\ldots$$ are the Jacobsthal numbers of the second kind \seqnum{A078008}.
The horizontal half is given by
$$H= V\cdot (1,f(x))=(1, xc(x))\cdot \left(\frac{1-x}{1-x-2x^2}, -x\right)\cdot \left(1, \frac{-x}{1+x}\right)$$ and thus we have the factorization
$$H= (1, xc(x)) \cdot \left(\frac{1-x}{1-x-2x^2},\frac{x}{1-x}\right).$$
This matrix has general term $\sum_{i=0}^{n-k}\binom{2n-i-1}{n-k-i}(-1)^i$, and begins
$$\left(
\begin{array}{ccccccc}
 1 & 0 & 0 & 0 & 0 & 0 & 0 \\
 0 & 1 & 0 & 0 & 0 & 0 & 0 \\
 2 & 2 & 1 & 0 & 0 & 0 & 0 \\
 6 & 7 & 4 & 1 & 0 & 0 & 0 \\
 22 & 24 & 16 & 6 & 1 & 0 & 0 \\
 80 & 86 & 62 & 29 & 8 & 1 & 0 \\
 296 & 314 & 239 & 128 & 46 & 10 & 1 \\
\end{array}
\right).$$
The sequence
$$1,0,2,6,22,80,296,\ldots$$ which is a Catalan image of the Jacobsthal numbers of the second kind, gives the central coefficients of the matrix $\left(\frac{1}{1-x},\frac{-x}{1+x}\right)$. The Hankel transform of this sequence begins
$$1, 2, 0, -8, -16, 0, 64, 128, 0, -512, -1024,\ldots$$ This sequence has generating function $\frac{1}{1-2x+4x^2}$.
\end{example}
In general, we have the following result when $f(x)=\frac{x}{1-x}$. 
\begin{proposition} The vertical half of the Riordan array $\left(g(x), \frac{x}{1-x}\right)$ is given by
$$V=\left(\frac{g(xc(x))}{c(x)\sqrt{1-4x}},  xc(x)\right).$$
The horizontal half of the Riordan array $\left(g(x), \frac{x}{1-x}\right)$ is given by
$$H=\left(\frac{g(xc(x))}{c(x)\sqrt{1-4x}},  xc(x)^2\right).$$
\end{proposition}
\begin{proof}
We have $\phi(x)=xc(x)$. Then $\phi'(x)=\frac{1}{\sqrt{1-4x}}$. Thus
$$\frac{g(\phi(x))x \phi'(x)}{\phi(x)}=\frac{g(xc(x))x}{\sqrt{1-4x}x c(x)}=\frac{g(xc(x))}{c(x)\sqrt{1-4x}}.$$
The first result follows from this. Now we have
$$H= V\cdot \left(1, \frac{x}{1-x}\right)=\left(\frac{g(xc(x))}{c(x)\sqrt{1-4x}}, \frac{xc(x)}{1-xc(x)}\right).$$
The result now follows since we have $\frac{1}{1-xc(x)}=c(x)$.
\end{proof}
\begin{example}
We consider the Riordan array $\left(1-2x, \frac{1}{1-x}\right)$ which begins
$$\left(
\begin{array}{cccccccc}
 1 & 0 & 0 & 0 & 0 & 0 & 0 & 0 \\
 -2 & 1 & 0 & 0 & 0 & 0 & 0 & 0 \\
 0 & -1 & 1 & 0 & 0 & 0 & 0 & 0 \\
 0 & -1 & 0 & 1 & 0 & 0 & 0 & 0 \\
 0 & -1 & -1 & 1 & 1 & 0 & 0 & 0 \\
 0 & -1 & -2 & 0 & 2 & 1 & 0 & 0 \\
 0 & -1 & -3 & -2 & 2 & 3 & 1 & 0 \\
 0 & -1 & -4 & -5 & 0 & 5 & 4 & 1 \\
\end{array}
\right).$$
Now $g(\phi(x))=1-2 \phi(x)=1-2xc(x)=\sqrt{1-4x}$, and hence
we have
$$H=\left(\frac{1}{c(x)}, xc(x)\right).$$
This matrix begins
$$\left(
\begin{array}{cccccccc}
 1 & 0 & 0 & 0 & 0 & 0 & 0 & 0 \\
 -1 & 1 & 0 & 0 & 0 & 0 & 0 & 0 \\
 -1 & 0 & 1 & 0 & 0 & 0 & 0 & 0 \\
 -2 & 0 & 1 & 1 & 0 & 0 & 0 & 0 \\
 -5 & 0 & 2 & 2 & 1 & 0 & 0 & 0 \\
 -14 & 0 & 5 & 5 & 3 & 1 & 0 & 0 \\
 -42 & 0 & 14 & 14 & 9 & 4 & 1 & 0 \\
 -132 & 0 & 42 & 42 & 28 & 14 & 5 & 1 \\
\end{array}
\right).$$
We note that
$$V^{-1}=\left(\frac{1}{1-x}, x(1-x)\right).$$
We have
$$H = \left(\frac{1}{c(x)}, xc(x)^2\right),$$ which begins
$$\left(
\begin{array}{cccccccc}
 1 & 0 & 0 & 0 & 0 & 0 & 0 & 0 \\
 -1 & 1 & 0 & 0 & 0 & 0 & 0 & 0 \\
 -1 & 1 & 1 & 0 & 0 & 0 & 0 & 0 \\
 -2 & 2 & 3 & 1 & 0 & 0 & 0 & 0 \\
 -5 & 5 & 9 & 5 & 1 & 0 & 0 & 0 \\
 -14 & 14 & 28 & 20 & 7 & 1 & 0 & 0 \\
 -42 & 42 & 90 & 75 & 35 & 9 & 1 & 0 \\
 -132 & 132 & 297 & 275 & 154 & 54 & 11 & 1 \\
\end{array}
\right).$$
We have
$$H^{-1}=\left(1+x, \frac{x}{(1+x)^2}\right).$$
We note that the Hankel transform of the sequence
$$1,-1,-1,-2,-5,-14,-42,\ldots$$ is the sequence
$$1,-2,3,-4,5,-6,-7,\ldots$$ with generating function $\frac{1}{(1+x)^2}$.
\end{example}
We now look at the general case of the parameterized Riordan array $\left(\frac{1}{1-rx}, \frac{x}{1-x}\right)$. We obtain that
$$V=\left(\frac{1+\sqrt{1-4x}}{\sqrt{1-4x}(2-r+r\sqrt{1-4x})}, x c(x)\right).$$
The first column of this matrix is then given by the polynomials that begin
$$1, r+1, r^2+2r+3, r^3 + 3r^2 + 6r + 10, \ldots$$ with coefficient array
$$\left(\frac{1+\sqrt{1-4x}}{2 \sqrt{1-4x}}, xc(x)\right),$$ which begins
$$\left(
\begin{array}{cccccccc}
 1 & 0 & 0 & 0 & 0 & 0 & 0 & 0 \\
 1 & 1 & 0 & 0 & 0 & 0 & 0 & 0 \\
 3 & 2 & 1 & 0 & 0 & 0 & 0 & 0 \\
 10 & 6 & 3 & 1 & 0 & 0 & 0 & 0 \\
 35 & 20 & 10 & 4 & 1 & 0 & 0 & 0 \\
 126 & 70 & 35 & 15 & 5 & 1 & 0 & 0 \\
 462 & 252 & 126 & 56 & 21 & 6 & 1 & 0 \\
 1716 & 924 & 462 & 210 & 84 & 28 & 7 & 1 \\
\end{array}
\right).$$
Thus
$$V=\left(\left(\frac{1+\sqrt{1-4x}}{2 \sqrt{1-4x}}, xc(x)\right).\frac{1}{1-rx}, xc(x)\right).$$
We note that the Hankel transform of the first column of $V$ is then given by
$$h_n = [x^n] \frac{1}{1-2x+(r-1)^2x^2}.$$
When $g(x)=1-rx$, we obtain that
$$V=\left(\frac{1-2rx+\sqrt{1-4x}}{2\sqrt{1-4x}}, x c(x)\right).$$
Then the first column of $V$ begins
$$1, 1 - r, 3 - 2r, 2(5 - 3r), 5(7 - 4r), 14(9 - 5r), 42(11 - 6r),\ldots.$$
The Hankel transform of this sequence is then given by
$$h_n = [x^n]\frac{1-(r^2-4r+2)x+(1-r)^2x^2}{(1+2(r-1)x+x^2)^2}.$$
\section{Two special matrices}
The matrix $\left(\frac{1+x}{1-x}, x\right)$ and its inverse, $\left(\frac{1-x}{1+x}, x\right)$ deserve special mention. These matrices begin, respectively,
$$\left(
\begin{array}{cccccc}
 1 & 0 & 0 & 0 & 0 & 0 \\
 2 & 1 & 0 & 0 & 0 & 0 \\
 2 & 2 & 1 & 0 & 0 & 0 \\
 2 & 2 & 2 & 1 & 0 & 0 \\
 2 & 2 & 2 & 2 & 1 & 0 \\
 2 & 2 & 2 & 2 & 2 & 1 \\
\end{array}
\right),\quad\text{and}\quad \left(
\begin{array}{cccccc}
 1 & 0 & 0 & 0 & 0 & 0 \\
 -2 & 1 & 0 & 0 & 0 & 0 \\
 2 & -2 & 1 & 0 & 0 & 0 \\
 -2 & 2 & -2 & 1 & 0 & 0 \\
 2 & -2 & 2 & -2 & 1 & 0 \\
 -2 & 2 & -2 & 2 & -2 & 1 \\
\end{array}
\right).$$ We now, respectively, take the inverse binomial and the binomial transform of these matrices, to obtain
$$M_1=B^{-1} \cdot \left(\frac{1+x}{1-x}, x\right)=\left(\frac{1+2x}{1+x}, \frac{x}{1+x}\right)\quad (\seqnum{A112466})$$ and
$$M_2=B \cdot  \left(\frac{1-x}{1+x}, x\right)=\left(\frac{1-2x}{1-x}, \frac{x}{1-x}\right)\quad (\seqnum{A112467}).$$
We obtain the matrices that begin
$$\left(
\begin{array}{cccccccc}
 1 & 0 & 0 & 0 & 0 & 0 & 0 & 0 \\
 1 & 1 & 0 & 0 & 0 & 0 & 0 & 0 \\
 -1 & 0 & 1 & 0 & 0 & 0 & 0 & 0 \\
 1 & -1 & -1 & 1 & 0 & 0 & 0 & 0 \\
 -1 & 2 & 0 & -2 & 1 & 0 & 0 & 0 \\
 1 & -3 & 2 & 2 & -3 & 1 & 0 & 0 \\
 -1 & 4 & -5 & 0 & 5 & -4 & 1 & 0 \\
 1 & -5 & 9 & -5 & -5 & 9 & -5 & 1 \\
\end{array}
\right)$$ and
$$\left(
\begin{array}{cccccccc}
 1 & 0 & 0 & 0 & 0 & 0 & 0 & 0 \\
 -1 & 1 & 0 & 0 & 0 & 0 & 0 & 0 \\
 -1 & 0 & 1 & 0 & 0 & 0 & 0 & 0 \\
 -1 & -1 & 1 & 1 & 0 & 0 & 0 & 0 \\
 -1 & -2 & 0 & 2 & 1 & 0 & 0 & 0 \\
 -1 & -3 & -2 & 2 & 3 & 1 & 0 & 0 \\
 -1 & -4 & -5 & 0 & 5 & 4 & 1 & 0 \\
 -1 & -5 & -9 & -5 & 5 & 9 & 5 & 1 \\
\end{array}
\right),$$ respectively. These have general elements given by
$$(-1)^{n-k}\left(\binom{n}{n-k}-2 \binom{n-1}{n-k-1}\right)\quad \text{and}\quad \left(\binom{n}{n-k}-2 \binom{n-1}{n-k-1}\right) $$ respectively.
The vertical and horizontal halves of these matrices begin, respectively,
$$V_1=\left(
\begin{array}{ccccccc}
 1 & 0 & 0 & 0 & 0 & 0 & 0 \\
 0 & 1 & 0 & 0 & 0 & 0 & 0 \\
 0 & -1 & 1 & 0 & 0 & 0 & 0 \\
 0 & 2 & -2 & 1 & 0 & 0 & 0 \\
 0 & -5 & 5 & -3 & 1 & 0 & 0 \\
 0 & 14 & -14 & 9 & -4 & 1 & 0 \\
 0 & -42 & 42 & -28 & 14 & -5 & 1 \\
\end{array}
\right),$$
$$H_1=\left(
\begin{array}{ccccccc}
 1 & 0 & 0 & 0 & 0 & 0 & 0 \\
 0 & 1 & 0 & 0 & 0 & 0 & 0 \\
 0 & -2 & 1 & 0 & 0 & 0 & 0 \\
 0 & 5 & -4 & 1 & 0 & 0 & 0 \\
 0 & -14 & 14 & -6 & 1 & 0 & 0 \\
 0 & 42 & -48 & 27 & -8 & 1 & 0 \\
 0 & -132 & 165 & -110 & 44 & -10 & 1 \\
\end{array}
\right),$$ 
$$V_2=\left(
\begin{array}{ccccccc}
 1 & 0 & 0 & 0 & 0 & 0 & 0 \\
 0 & 1 & 0 & 0 & 0 & 0 & 0 \\
 0 & 1 & 1 & 0 & 0 & 0 & 0 \\
 0 & 2 & 2 & 1 & 0 & 0 & 0 \\
 0 & 5 & 5 & 3 & 1 & 0 & 0 \\
 0 & 14 & 14 & 9 & 4 & 1 & 0 \\
 0 & 42 & 42 & 28 & 14 & 5 & 1 \\
\end{array}
\right),$$ and 
$$H_2=\left(
\begin{array}{ccccccc}
 1 & 0 & 0 & 0 & 0 & 0 & 0 \\
 0 & 1 & 0 & 0 & 0 & 0 & 0 \\
 0 & 2 & 1 & 0 & 0 & 0 & 0 \\
 0 & 5 & 4 & 1 & 0 & 0 & 0 \\
 0 & 14 & 14 & 6 & 1 & 0 & 0 \\
 0 & 42 & 48 & 27 & 8 & 1 & 0 \\
 0 & 132 & 165 & 110 & 44 & 10 & 1 \\
\end{array}
\right).$$ 
\begin{proposition} For the Riordan array $M_2=\left(\frac{1-2x}{1-x}, \frac{x}{1-x}\right)$ we have 
$$V_2 = (1, xc(x))$$ and 
$$H_2 = (1, xc(x)^2).$$
\end{proposition}
\begin{proof}
We have $f(x)=\frac{x}{1-x}$ and so $\phi(x)=\text{Rev}(x(1-x))=xc(x)$. 
Starting from $g(x)=\frac{1-2x}{1-x}$, we have 
\begin{align*}
\frac{g(xc(x))x (xc(x))'}{xc(x)}&=\frac{1-2xc(x)}{1-xc(x)}\cdot \frac{1}{1-4x}\cdot \frac{x}{xc(x)}\\
&=(1-2xc(x))c(x)\frac{1}{\sqrt{1-4x}}\frac{1}{c(x)}\\
&=\frac{ \sqrt{1-4x}c(x)}{\sqrt{1-4x}c(x)}\\
&=1.\end{align*}
Thus we have 
$$V_2 = (1, xc(x)).$$
Now 
$$H_2 = (1, xc(x))\cdot \left(1, \frac{x}{1-x}\right)=\left(1, \frac{xc(x)}{1-xc(x)}\right)=(1, xc(x)^2).$$
\end{proof}
In similar fashion, we can show that 
$$ V_1 = (1, -xc(x)) \quad \text{and}\quad H_1=(1,-xc(x)^2).$$

\section{The Catalan matrix $(c(x)^2, xc(x)^2)$}
In this section, we seek the Riordan array whose vertical half $V$ satisfies
$$V =\left(\frac{g(\phi(x))x \phi'(x)}{\phi(x)}, \phi(x)\right)= (c(x)^2, xc(x)^2).$$
We have $$\phi(x)=\text{Rev}\left(\frac{x^2}{f(x)}\right)=xc(x)^2=\text{Rev}\left(\frac{x}{(1+x)^2}\right).$$
From this we deduce that
$$f(x)=x(1+x)^2.$$
Now we require that
$$\frac{g(\phi(x))x \phi'(x)}{\phi(x)}=c(x)^2.$$
We can solve this equation for $g(x)$ by substituting $\bar{\phi}(x)$ for $x$, since then $g(\phi(\bar{\phi}(x)))=g(x)$.
Now $c(\bar{\phi}(x))=c\left(\frac{x}{(1+x)^2}\right)=1+x$ and
$$\sqrt{1-4 \frac{x}{(1+x)^2}}=\frac{1-x}{1+x}.$$
Now
$$\phi'(x)=\frac{(1-\sqrt{1-4x})^2}{4x^2 \sqrt{1-4x}},$$ and so
$$\phi'(\bar{\phi}(x))=\frac{(1+x)^3}{1-x}.$$
Thus we have
$$g(x)\frac{(1+x)^3}{1-x} \frac{1}{(1+x)^2}=(1+x)^2.$$
We deduce that $$g(x)=(1+x)(1-x)=1-x^2.$$
We thus obtain the following result.
\begin{proposition} The vertical half $V$ of the Riordan array
$$(1-x^2, x(1+x)^2)$$ is the Catalan matrix $$V=(c(x)^2, xc(x)^2).$$ The corresponding horizontal half $H$ is given by $$H=(c(x)^2, xc(x)^4).$$
\end{proposition}
The Riordan array $(1-x^2, x(1+x)^2)$ begins
$$\left(
\begin{array}{cccccccc}
 1 & 0 & 0 & 0 & 0 & 0 & 0 & 0 \\
 0 & 1 & 0 & 0 & 0 & 0 & 0 & 0 \\
 -1 & 2 & 1 & 0 & 0 & 0 & 0 & 0 \\
 0 & 0 & 4 & 1 & 0 & 0 & 0 & 0 \\
 0 & -2 & 5 & 6 & 1 & 0 & 0 & 0 \\
 0 & -1 & 0 & 14 & 8 & 1 & 0 & 0 \\
 0 & 0 & -5 & 14 & 27 & 10 & 1 & 0 \\
 0 & 0 & -4 & 0 & 48 & 44 & 12 & 1 \\
\end{array}
\right).$$
The Catalan matrix $(c(x)^2, xc(x)^2)$ begins
$$\left(
\begin{array}{cccccccc}
 1 & 0 & 0 & 0 & 0 & 0 & 0 & 0 \\
 2 & 1 & 0 & 0 & 0 & 0 & 0 & 0 \\
 5 & 4 & 1 & 0 & 0 & 0 & 0 & 0 \\
 14 & 14 & 6 & 1 & 0 & 0 & 0 & 0 \\
 42 & 48 & 27 & 8 & 1 & 0 & 0 & 0 \\
 132 & 165 & 110 & 44 & 10 & 1 & 0 & 0 \\
 429 & 572 & 429 & 208 & 65 & 12 & 1 & 0 \\
 1430 & 2002 & 1638 & 910 & 350 & 90 & 14 & 1 \\
\end{array}
\right).$$
This is \seqnum{A039598}. We note that this Bell matrix satisfies
$$(c(x)^2, xc(x)^2)=\left(\frac{1}{(1+x)^2}, \frac{x}{(1+x)^2}\right)^{-1}.$$
The corresponding horizontal half is then given by
$$H = (c(x)^2, xc(x)^4),$$ since we have
$$H= (c(x)^2, xc(x)^2)\cdot (1, x(1+x)^2)=(c(x^2), xc(x)^2(1+xc(x)^2)^2)=(c(x)^2, xc(x)^4).$$
This matrix begins
$$\left(
\begin{array}{cccccccc}
 1 & 0 & 0 & 0 & 0 & 0 & 0 & 0 \\
 2 & 1 & 0 & 0 & 0 & 0 & 0 & 0 \\
 5 & 6 & 1 & 0 & 0 & 0 & 0 & 0 \\
 14 & 27 & 10 & 1 & 0 & 0 & 0 & 0 \\
 42 & 110 & 65 & 14 & 1 & 0 & 0 & 0 \\
 132 & 429 & 350 & 119 & 18 & 1 & 0 & 0 \\
 429 & 1638 & 1700 & 798 & 189 & 22 & 1 & 0 \\
 1430 & 6188 & 7752 & 4655 & 1518 & 275 & 26 & 1 \\
\end{array}
\right).$$
The row sums of this matrix \seqnum{A125187}$(n+1)$ have an interesting property, pointed out by Michael Somos.
The sequence begins
$$ 1, 3, 12, 52, 232, 1049, 4777, 21845,\ldots.$$
Its Hankel transform is the even bisection of the Fibonacci numbers $F_{2(n+1)} \seqnum{A001906}$
$$1, 3, 8, 21, 55, 144, 377, 987, 2584, 6765, 17711,\ldots.$$
The Hankel transform of this sequence with a $1$ preprended (this is \seqnum{A125187}) has Hankel transform given by the odd bisection of the Fibonacci numbers $F_{2n+1}$ \seqnum{A001519},
$$1, 2, 5, 13, 34, 89, 233, 610, 1597, 4181, 10946,\ldots.$$
The row sum sequence has generating function
$$\frac{c(x)^2}{1-x c(x)^4}=\frac{(1-x)(1-2x)+(1+x)\sqrt{1-4x}}{2(1-5x+2x^2-x^3)}.$$
The row sums of $V=(c(x)^2, xc(x)^2)$ are given by $\binom{2n+1}{n+1}$, \seqnum{A001700}. This sequence has a Hankel transform given by
$$1,1,1,\ldots$$ while the Hankel transform of this sequence with a prepended $1$ is given by $n+1$.
The matrix $(1+x, x(1+x)^2)$ has its vertical half given by the Riordan array
$$\left(\frac{c(x)^2}{1-xc(x)^2}, xc(x)^2\right).$$ This matrix begins
$$\left(
\begin{array}{ccccccc}
 1 & 0 & 0 & 0 & 0 & 0 & 0 \\
 3 & 1 & 0 & 0 & 0 & 0 & 0 \\
 10 & 5 & 1 & 0 & 0 & 0 & 0 \\
 35 & 21 & 7 & 1 & 0 & 0 & 0 \\
 126 & 84 & 36 & 9 & 1 & 0 & 0 \\
 462 & 330 & 165 & 55 & 11 & 1 & 0 \\
 1716 & 1287 & 715 & 286 & 78 & 13 & 1 \\
\end{array}
\right).$$
We finish by noting the following.
\begin{proposition} The Riordan array $(1, x(1+x)^2)$ has its vertical half $V$ given by 
$$V = \left(\frac{1}{\sqrt{1-4x}}, xc(x)^2\right)$$ and its horizontal half $H$ given by 
$$H = \left(\frac{1}{\sqrt{1-4x}}, xc(x)^4\right).$$ 
\end{proposition}
Now the array $(1, x(1+x)^2)$ has general term $\binom{2k}{n-k}$. It begins 
$$\left(
\begin{array}{ccccccc}
 1 & 0 & 0 & 0 & 0 & 0 & 0 \\
 0 & 1 & 0 & 0 & 0 & 0 & 0 \\
 0 & 2 & 1 & 0 & 0 & 0 & 0 \\
 0 & 1 & 4 & 1 & 0 & 0 & 0 \\
 0 & 0 & 6 & 6 & 1 & 0 & 0 \\
 0 & 0 & 4 & 15 & 8 & 1 & 0 \\
 0 & 0 & 1 & 20 & 28 & 10 & 1 \\
\end{array}
\right).$$ 
The general term of $V$ is then $\binom{2n}{n-k}$. This array begins
$$\left(
\begin{array}{ccccccc}
 1 & 0 & 0 & 0 & 0 & 0 & 0 \\
 2 & 1 & 0 & 0 & 0 & 0 & 0 \\
 6 & 4 & 1 & 0 & 0 & 0 & 0 \\
 20 & 15 & 6 & 1 & 0 & 0 & 0 \\
 70 & 56 & 28 & 8 & 1 & 0 & 0 \\
 252 & 210 & 120 & 45 & 10 & 1 & 0 \\
 924 & 792 & 495 & 220 & 66 & 12 & 1 \\
\end{array}
\right).$$ This is \seqnum{A094527}. 
The horizontal half $H$ has general term $\binom{2(n+k)}{n-k}$ and it begins
$$\left(
\begin{array}{ccccccc}
 1 & 0 & 0 & 0 & 0 & 0 & 0 \\
 2 & 1 & 0 & 0 & 0 & 0 & 0 \\
 6 & 6 & 1 & 0 & 0 & 0 & 0 \\
 20 & 28 & 10 & 1 & 0 & 0 & 0 \\
 70 & 120 & 66 & 14 & 1 & 0 & 0 \\
 252 & 495 & 364 & 120 & 18 & 1 & 0 \\
 924 & 2002 & 1820 & 816 & 190 & 22 & 1 \\
\end{array}
\right).$$
An interesting feature of the matrix $V=\binom{2n}{n-k}$ is the following. If we multiply it by the Riordan array $(c(x), x)$ on the left, to get 
$$(c(x), x) \cdot \left(\frac{1}{\sqrt{1-4x}}, xc(x)^2\right)=\left(\frac{c(x)}{\sqrt{1-4x}}, xc(x)^2\right)$$ then the row sums of this matrix have generating function 
$$\frac{\frac{c(x)}{\sqrt{1-4x}}}{1-xc(x)^2}=\frac{1}{1-4x}.$$ 
Thus this new matrix has row sums given by $4^n$. 
$$\left(
\begin{array}{cccccc}
 1 & 0 & 0 & 0 & 0 & 0 \\
 1 & 1 & 0 & 0 & 0 & 0 \\
 2 & 1 & 1 & 0 & 0 & 0 \\
 5 & 2 & 1 & 1 & 0 & 0 \\
 14 & 5 & 2 & 1 & 1 & 0 \\
 42 & 14 & 5 & 2 & 1 & 1 \\
\end{array}
\right)\left(
\begin{array}{cccccc}
 1 & 0 & 0 & 0 & 0 & 0 \\
 2 & 1 & 0 & 0 & 0 & 0 \\
 6 & 4 & 1 & 0 & 0 & 0 \\
 20 & 15 & 6 & 1 & 0 & 0 \\
 70 & 56 & 28 & 8 & 1 & 0 \\
 252 & 210 & 120 & 45 & 10 & 1 \\
\end{array}
\right)=\left(
\begin{array}{cccccc}
 1 & 0 & 0 & 0 & 0 & 0 \\
 3 & 1 & 0 & 0 & 0 & 0 \\
 10 & 5 & 1 & 0 & 0 & 0 \\
 35 & 21 & 7 & 1 & 0 & 0 \\
 126 & 84 & 36 & 9 & 1 & 0 \\
 462 & 330 & 165 & 55 & 11 & 1 \\
\end{array}
\right).$$ We obtain the following.
\begin{proposition} The row sums of the vertical half of the Riordan array $(1+x, x(1+x)^2)$ are given by $4^n$.
\end{proposition}
We deduce that 
$$\sum_{k=0}^n \sum_{j=0}^n C_{n-j}\binom{2j}{j-k} = 4^n.$$
The related sum $\sum_{k=0}^n \sum_{j=0}^n \binom{2(n-j)}{n-j}\binom{2j}{j-k}$ appears to give \seqnum{A258431}.

\bigskip
\hrule

\noindent 2010 {\it Mathematics Subject Classification}: Primary
15B36; Secondary 11B83, 11C20.
\noindent \emph{Keywords:} Riordan array, Vertical half, Horizontal half, Catalan matrix.

\bigskip
\hrule
\bigskip
\noindent (Concerned with sequences
\seqnum{A000108},
\seqnum{A001316},
\seqnum{A001519},
\seqnum{A001700},
\seqnum{A001906},
\seqnum{A007318},
\seqnum{A033184},
\seqnum{A039598},
\seqnum{A063886},
\seqnum{A078008},
\seqnum{A094527},
\seqnum{A106566},
\seqnum{A112466},
\seqnum{A112467},
\seqnum{A125187},
\seqnum{A125187}, 
\seqnum{A128899}, and
\seqnum{A258431}.)

\end{document}